\documentclass{amsart}

\usepackage{graphics}
\usepackage{latexsym}
\usepackage{amssymb}
 \usepackage{amsmath}
\usepackage{amsfonts}

\newtheorem{thm}{Theorem}
 \newtheorem{prop}{Proposition}

\newtheorem{cor}{Corollary}
\newtheorem{lem}{Lemma}

\newtheorem{con}{Conjecture}

\begin{document}  
 
\title{Tensor Products of irreducible representations of the group   $\ GL(3,{\mathbb{F}}_{q})$    }

 \author{Luisa Aburto-Hageman}
\address{Instituto de Matem\'atica, Pontificia Universidad Cat\'olica de Valpara\'{i}so, Casilla 4059, Valpara\'{i}so, Chile}
\email{laburto@ucv.cl}

 \author{Jos\'e Pantoja}
\address{Instituto de Matem\'atica, Pontificia Universidad Cat\'olica de Valpara\'{i}so, Casilla 4059, Valpara\'{i}so, Chile}
\email{jpantoja@ucv.cl}

\author{Jorge Soto-Andrade}
\address{Departamento de Matem\'aticas, Facultad de Ciencias, Universidad de Chile, Casilla 653, Santiago, Chile }
\email{sotoandr@uchile.cl}

 \thanks{The first and second authors were partially supported by Pontificia
Universidad Cat\'{o}lica de Valpara\'{\i}so. The second and third authors were partially supported by FONDECYT Grants 1040444, 1070246 and by PICS CNRS 1413}
 
\begin{abstract}
We describe the tensor products of two irreducible linear complex representations of the group $G = GL(3,{\mathbb{F}}_{q})$ in terms of  induced representations by linear characters of maximal torii and also in terms of classical and generalized Gelfand-Graev representations. Our results include MacDonald's conjectures for $G$ and at the same time they are extensions to  $G$ of finite counterparts to classical  results on tensor products of holomorphic and anti-holomorphic representations of the group $SL(2, \mathbb R) $. Moreover they provide an easy way to  decompose these tensor products, with the help of Frobenius reciprocity.  We also state some conjectures for the general case of $GL(n,{\mathbb{F}}_{q})$.
\end{abstract}
\maketitle

\section{Introduction}
 This paper studies the tensor product of two irreducible linear complex
representations of the group $G=GL(3,{\mathbb{F}}_{q})$, generalizing
previous work \cite{AHP} of the first two named authors on $GL(2,{\mathbb{F}}_{q})$ and 
$SL(2,{\mathbb{F}}_{q})$. 

Our main idea is to express tensor products of irreducible representations
in terms of induced representations. Indeed, this paper adds to the
experimental evidence which strongly suggests that, for classical groups,
tensor products of generic irreducible representations are essentially
induced representations by suitable linear characters from either of the
involved torii. Here ``essentially" means ``up to lower dimensional correcting
terms".

General results of this type are of interest for several reasons.

First, for finite classical groups, tensor products of irreducible
representations realized as induced representations of this kind, have
appeared already in the context of the famous MacDonald's conjectures, which
state, in the case of $GL(n,{\mathbb{F}}_{q})$, that the tensor product of
the canonical Steinberg representation $St$ with a generic irreducible
representation $R_{T}(\theta )$, associated to a character $\theta ,$ in
general position, of a maximal torus $T$ of $GL(n,{\mathbb{F}}_{q}),$ \
equals the representation induced by $\theta $ from $T$ \ to $\ GL(n,{%
\mathbb{F}}_{q})$. These conjectures were proved much later by Deligne and
Lusztig as a corollary of their construction in \cite {DL}.

Second, it was pointed out to us by A. Guichardet, that remarkably enough theoretical physicists, like Rideau \cite{R}, have
been interested on their own in describing tensor products of irreducible
unitary representations of a same series for classical real Lie groups like $%
SL(2,{\mathbb{R}}),$ proving that the tensor product of holomorphic and
antiholomorphic discrete series representations of $SL(2,{\mathbb{R}})$ is
given by a suitable induced representation from the corresponding torus. Rideau's results 
were  extended to principal series representations as well by Guichardet, and we realized that analogous results would hold for $GL(2,{\mathbb{F}}_{q})$ if adequate correcting terms were
introduced. So in fact, Rideau's results and MacDonald's conjectures dwell
under the same roof. A complete description of tensor products of
irreducible representations of $GL(2,{\mathbb{F}}_{q})$ and $SL(2,{\mathbb{F}
}_{q})$ in terms of induced representations, was then given in \cite{AHP}.

Third, tensor products of irreducible representations may be decomposed
quite easily, via Frobenius reciprocity, once you have described them in
terms of induced representations (see section 4 below).  

Fourth, the realization of the tensor product of two irreducible
representations as an induced representation from a linear character also
allows us to guess and glean interesting relations between special functions
of various sorts. \ Indeed, such an induced representation may be looked
upon as a "twisted" natural representation, for which spherical functions
may be calculated, in the multiplicity-free case, as in \cite{jaja}. For
instance, recent work \cite{K} on the relationship between classical Kloosterman
sums and the so called Legendre sums and Soto-Andrade sums for  $
G=PGL(2,q)$, may be understood as a consequence of the fact that  $$   
  St\otimes St= (\underset{T\,\,\uparrow
\,G}{Ind}{\mathbf{1)\; \oplus \;}}St $$
where \ $T$ \ denotes the anisotropic torus of $G,$ \ if you recall that
this induced representation is just the natural multiplicity-free
representation of $G$ \ associated to its homographic action on the (double
cover of) finite Poincar\'{e}'s upper half plane, whose spherical functions
are given by the last two aforementioned sums  \cite {jaja}. Also
Gel'fand-Graev representations appear in this way (see subsection 3.2 below).

Fifth, there is, on the other hand, a non obvious but close connections
between tensor products and Gelfand models for the classical groups. Recall
that a Gelfand model for a group $G$ is any representation of \ $G$ \ which
decomposes with multiplicity one as the sum of all the irreducible
representations of  $G$.

Intriguing experimental evidence suggests that quite often Gelfand models or
"quasi-models" may be obtained as tensor products of Steinberg
representations. The first case is $ G=PGL(2,q),$ where the tensor square
of the Steinberg representation affords a quasi-model for $G$, where only
the sign representation is missing. Analogous results seem to hold for $
PGL(n,q).$ It is indeed very interesting to realize a Gelfand model or
quasi-model as an induced representation from a linear character, specially
from the unit character. In the latter case we say that we have a geometrical
Gelfand model or quasi-model.

Sixth, although the problem of decomposing tensor products in the p-adic
case was treated in \cite{AR1,AR2,T}, apparently tensor products of
two different series of representations have not been studied, and the
description of this tensor products as induced representation by linear
characters from torii has not been worked out.

So, in this paper we describe tensor products of two irreducible
representations of the group $G=GL(3,{\mathbb{F}}_{q})$ as induced
representations by linear characters from torii of $G$, up to adequate lower
dimensional correcting terms. We notice that the situation is more complex
than in dimension $2,$ and subtler in some cases (see for example case 10  of
theorem 1 below). We concentrate here on "generic" irreducible representations,
since the "non generic" or "degenerate" lower dimensional irreducible
representations may be obtained from limiting cases of the generic ones, as
in lemma 2. We show then how to get formulas for the non generic\
representations from the generic ones.

Depending on the series of representations to which an irreducible
representation belongs, a torus is associated in a natural way, and it is
the one we use in the corresponding formula. Thus in the cases where we deal
with the tensor product of two irreducible representations belonging to two
different series, we give two descriptions, one in terms of each torus.

However, another approach is also possible in this case: to express our
tensor product in terms of the intersection of the two different torii (the center $Z(G)$ of $G$) times the upper unipotent subgroup $N$. More
specifically, let $\alpha $ be a character of    $k^{\times }=Z\left(
G\right) =Z$ and let $\psi $ be a non trivial character of $k^{\times }$
seen as character on $\ N$ \ by 
$$\left( 
\begin{array}{ccc}
1 & x & z \\ 
0 & 1 & y \\ 
0 & 0 & 1
\end{array}
\right) \longmapsto \psi \left( x+y\right). $$ 
\noindent
The classical
Gelfand-Graev representation of $\ G$ is the representation 
$Ind_{ZN}^{G}\left( \alpha \psi \right).$ We call  {\em generalized
Gelfand-Graev representation}  of $G,$ any                                                                                                                                                                                                                                                                                                                                                                                                                                                                                                                                                                                                                                                                                                                                                                                                                                                                                                                                                                                                                                                                                                                                                                                                                                                                                                                                                representation of the form $Ind_{ZN_{1}}^{G}
\left( \alpha \psi \right),$ where $N_{1}$ is the subgroup of $N$ defined by the condition 
$x=0.$ Proposition 1 below describes the tensor product of a discrete series
representation and a principal series representation of $G$ \ as the sum of
these two Gelfand-Graev representations, classical and generalized.

A similar result should be obtained in the general case of dimension $n$. We
state this as a conjecture, after proposition 1 below.

The paper is organized as follows: section 2 gives the preliminaries, mainly  the
character table of the group $G$; section 3 is devoted to our main results.

\section{Preliminaries}
\label{sec:prelim}
\subsection{Notations}
\label{subsec:not}
Let ${\mathbb{F}}_{q}$ be the finite field of $q=p^{n}$ elements, $p$ prime.
Then the quadratic and cubic extensions of ${\mathbb{F}}_{q}$ are ${\mathbb{F
}}_{q^{2}}$ and ${\mathbb{F}}_{q^{3}}$ respectively. We identify $z\in {
\mathbb{F}}_{q^{3}}$ with the ${\mathbb{F}}_{q}$-automorphism of $\ {\mathbb{
F}}_{q^{3}}$ given by $x\longmapsto zx$, and we identify $w\in {\mathbb{F}}
_{q^{2}}$ with the ${\mathbb{F}}_{q}$-automorphism of ${\mathbb{F}}_{q^{2}}$
given by $x\longmapsto wx$. We also denote by $z$ the matrix of the above
automorphism with respect to the basis $\left \{1,\sigma ,\sigma
^{2}\right
\},$ $\sigma $ a generator of ${\mathbb{F}}_{q^{3}}$ over ${
\mathbb{F}}_{q}.$ Similarly, for $w$ in ${\mathbb{F}}_{q^{2}}$, $w$ also
denotes the matrix of the automorphism $x\longmapsto wx$ with respect to the
basis $\left \{1,\tau \right \},$ $\tau $ a generator of ${\mathbb{F}}
_{q^{2}}$ over ${\mathbb{F}}_{q}$.

The above defines a monomorphism of ${\mathbb{F}}_{q^{3}}^{\times }$ into $G$
whose image is the anisotropic torus $T_{a}$.

We denote by $T_{i}$ the isotropic torus of $G$; thus $T_{i}$ is the image
of ${\mathbb{F}}_{q}^{\times }\times {\mathbb{F}}_{q}^{\times }\times {
\mathbb{F}}_{q}^{\times }$ in $G.$ Similarly we denote by $T_{m}$ the
intermediate torus $$T_{m}=\left\{ \left( 
\begin{array}{cc}
x & 0 \\ 
0 & a
\end{array}
\right) \quad \mid \quad a\in {\mathbb{F}}_{q}^{\times },\;x\in {\mathbb{F}}
_{q^{2}}^{\times }\right\} .$$
 Finally, we set $L$ equal to the intermediate
Levi subgroup $GL(2,{\mathbb{F}}_{q})\times {\mathbb{F}}_{q}^{\times }$. 

We
adopt the convention that $\eta $ denotes a representation of $G$ as well as
its character.

\subsection{The conjugacy classes of  $G = \ GL(3,{\mathbb{F}}_{q})$}
\label{subsec:conjugacy}

The elements

$T^{a}=\left( 
\begin{array}{lll}
a & 0 & 0 \\ 
0 & a & 0 \\ 
0 & 0 & a
\end{array}
\right) ;\quad T_{1}^{a}=\left( 
\begin{array}{lll}
a & 1 & 0 \\ 
0 & a & 0 \\ 
0 & 0 & a
\end{array}
\right) ;\quad T_{11}^{a}=\left( 
\begin{array}{lll}
a & 1 & 0 \\ 
0 & a & 1 \\ 
0 & 0 & a
\end{array}
\right) $

$T^{ab}=\left( 
\begin{array}{lll}
a & 0 & 0 \\ 
0 & a & 0 \\ 
0 & 0 & b%
\end{array}%
\right) ;\quad T_{1}^{ab}=\left( 
\begin{array}{lll}
a & 1 & 0 \\ 
0 & a & 0 \\ 
0 & 0 & b%
\end{array}%
\right) ;\quad T^{abc}=\left( 
\begin{array}{lll}
a & 0 & 0 \\ 
0 & b & 0 \\ 
0 & 0 & c%
\end{array}%
\right) $

\noindent
(where $a\neq b$, $a\neq c$, $b\neq c$) are representatives of different
conjugacy classes in $G$. In addition, if we set

$T^{\varkappa a}=\left( 
\begin{array}{cc}
\varkappa & 0 \\ 
0 & a%
\end{array}%
\right) ;\quad T^{z}=z$ \ (where $\ a\in {\mathbb{F}}_{q}^{\times
},\,\varkappa \in {\mathbb{F}}_{q^{2}}^{\times }$ and $z\in {\mathbb{F}}%
_{q^{3}}^{\times }$), then the set of $3$  by $3$ matrices $\left\{
T^{a},T_{1}^{a},T_{11}^{a},T^{ab},T_{1}^{ab},T^{abc},T^{\varkappa
a},T^{z}\right\} $ is {\em a full set of representatives of the conjugacy classes
of $G$.}  

Moreover the following holds.
 
\begin{lem}
  We have 
\begin{enumerate}
\item For $T=T_{1}^{a},T_{11}^{a},T_{1}^{ab},$ we have $\left\{ X\in G
\mid XTX^{-1}\in T_{i}\right\} =\phi $ and also

$\left\{ X\in G \mid XT^{\varkappa a}X^{-1}\in T_{i}\right\}
=\left\{ X\in G \mid  XT^{z}X^{-1}\in T_{i}\right\} =\phi .$

\item $\left\{ X\in G \mid  XT^{ab}X^{-1}\in T_{i}\right\} =$

$\ \ \ $

$
\begin{array}{cc}
= & \left\{ \left( 
\begin{array}{ccc}
r & s & 0 \\ 
t & u & 0 \\ 
0 & 0 & k
\end{array}
\right) ,\left( 
\begin{array}{ccc}
k & 0 & 0 \\ 
0 & r & s \\ 
0 & t & u
\end{array}
\right) \left( 
\begin{array}{lll}
0 & 0 & 1 \\ 
0 & 1 & 0 \\ 
1 & 0 & 0
\end{array}
\right) ,\qquad \qquad \qquad \qquad \qquad \right.\\  
& \left. \left( 
\begin{array}{lll}
r & 0 & s \\ 
0 & k & 0 \\ 
t & 0 & u%
\end{array}%
\right) \left( 
\begin{array}{lll}
1 & 0 & 0 \\ 
0 & 0 & 1 \\ 
0 & 1 & 0%
\end{array}%
\right)   \mid   k\in {\mathbb{F}}_{q}^{\times },\left( 
\begin{array}{ll}
r & s \\ 
t & u%
\end{array}%
\right) \in GL(2,{\mathbb{F}}_{q})\right\}%
\end{array}%
$
\noindent 
We note that $\left( 
\begin{array}{ccc}
0 & 0 & k \\ 
r & s & 0 \\ 
t & u & 0%
\end{array}%
\right) =\left( 
\begin{array}{ccc}
k & 0 & 0 \\ 
0 & r & s \\ 
0 & t & u%
\end{array}%
\right) \left( 
\begin{array}{lll}
0 & 0 & 1 \\ 
0 & 1 & 0 \\ 
1 & 0 & 0%
\end{array}%
\right) $ and

$\left( 
\begin{array}{lll}
r & s & 0 \\ 
0 & 0 & k \\ 
t & u & 0%
\end{array}%
\right) =\left( 
\begin{array}{lll}
r & 0 & s \\ 
0 & k & 0 \\ 
t & 0 & u%
\end{array}%
\right) \left( 
\begin{array}{lll}
1 & 0 & 0 \\ 
0 & 0 & 1 \\ 
0 & 1 & 0%
\end{array}%
\right) $

\item $\left\{ X\in G \mid XT^{abc}X^{-1}\in T_{i}\right\}
=T_{i}\times S_{3}$

\item For $T=T_{1}^{a},T_{11}^{a},T^{ab},T_{1}^{ab}$ we get $\left\{ X\in
G \mid  XTX^{-1}\in T_{a}\right\} =\phi $;  moreover

  $\left\{ X\in G \mid  XT^{abc}X^{-1}\in T_{a}\right\} =\phi $
and   $\left\{ X\in G \mid  XT^{\varkappa a}X^{-1}\in T_{a}\right\}
=\phi $

\item $\left\{ X\in G \mid  XT^{z}X^{-1}\in T_{a}\right\}
=T_{a}\times \Gamma _{3}$ 

\noindent
where $\Gamma _{3}$ is the Galois group of the
cubic extension generated by Frobenius automorphism, acting naturally on $%
T_{a}$.

\item For $T=T_{1}^{a},T_{11}^{a},T_{1}^{ab},$ we get $\left\{ X\in G
\mid  XTX^{-1}\in T_{m}\right\} =\phi $, and also

$\left\{ X\in G \mid  XT^{z}X^{-1}\in T_{m}\right\} =\phi $

\item $\left\{ X\in G \mid  XT^{ab}X^{-1}\in T_{m}\right\} =L$

\item $\left\{ X\in G \mid  XT^{\varkappa a}X^{-1}\in
T_{m}\right\} =T_{m}\times \left\{ \left( 
\begin{array}{cc}
\Gamma _{2} & 0 \\ 
0 & 1
\end{array}
\right) \right\} $\\
where $\Gamma _{2}$ is (isomorphic to) the Galois group of the quadratic
extension of the field ${\mathbb{F}}_{q}$, acting through Frobenius
automorphism of the quadratic extension.
\end{enumerate}
\end{lem}

 \subsection {The characters of   $G = \ GL(3,{\mathbb{F}}_{q})$}
\label{subsec: 2.3}
The following table gives the irreducible characters of the group $G.$ This
result leans on the work of Steinberg \cite{S}.\\ 
  {\bf   Character Table of  $  GL(3, \mathbb{F}_{ q})$}\\  
\begin{tabular}{|c|c|c|c|c|}
\hline
\ $
\begin{array}{l}
{\text{Elementary divisors of }} \\ 
{\text{conjugacy classes}}
\end{array}
$ & $
\begin{array}{c}
{\text{Represen-}} \\ 
{\text{tatives}}
\end{array}
$ & ${\mathbf{\chi }}_{\alpha }^{1}$ & ${\mathbf{\chi }}_{\alpha }^{q^{2}+q}$
& ${\mathbf{\chi }}_{\alpha }^{q^{3}}$ \\ \hline
\ $x-a,x-a,x-a$ & $T^{a}$ & $\alpha ^{3}\left( a\right) $ & $\left(
q^{2}+q\right) \alpha ^{3}\left( a\right) $ & $q^{3}\alpha ^{3}\left(
a\right) $ \\ \hline
\ $1,x-a,\left( x-a\right) ^{2}$ & $T_{1}^{a}$ & $\alpha ^{3}\left( a\right) 
$ & $q\alpha ^{3}\left( a\right) $ & $0$ \\ 
\ $1,1,\left( x-a\right) ^{3}$ & $T_{11}^{a}$ & $\alpha ^{3}\left( a\right) $
& $0$ & $0$ \\ \hline
\ $1,x-a,\left( x-a\right) \left( x-b\right) $ & $T^{ab}$ & $\alpha \left(
a^{2}b\right) $ & $\left( q+1\right) \alpha \left( a^{2}b\right) $ & $%
q\alpha \left( a^{2}b\right) $ \\ \hline
\ $1,1,\left( x-a\right) ^{2}\left( x-b\right) $ & $T_{1}^{ab}$ & $\alpha
\left( a^{2}b\right) $ & $\alpha \left( a^{2}b\right) $ & $0$ \\ \hline  
\ $1,1,\left( x-a\right) \left( x-b\right) \left( x-c\right) $ & $T^{abc}$ & 
$\alpha \left( abc\right) $ & $2\alpha \left( abc\right) $ & $\alpha \left(
abc\right) $ \\ \hline
\ $1,1,\left( x-a\right) \left( x-\varkappa \right) \left( x-\varkappa
^{q}\right) $ & $T^{\varkappa a}$ & $\alpha \left( a\varkappa \varkappa
^{q}\right) $ & $0$ & $-\alpha \left( a\varkappa \varkappa ^{q}\right) $ \\ \hline
\ $1,1,\left( x-z\right) \left( x-z^{q}\right) \left( x-z^{q^{2}}\right) $ & 
$T^{z}$ & $\alpha \left( zz^{q}z^{q^{2}}\right) $ & $-\alpha \left(
zz^{q}z^{q^{2}}\right) $ & $\alpha \left( zz^{q}z^{q^{2}}\right) $ \\ \hline
\ $
\begin{array}{l}
a,b,c\in {\mathbb{F}}_{q}^{\times },a\neq b\neq c\neq a \\ 
\varkappa \in {\mathbb{F}}_{q^{2}}-{\mathbb{F}}_{q},z\in {\mathbb{F}}%
_{q^{3}}-{\mathbb{F}}_{q}%
\end{array}%

$ &  &  &  &  \\ \hline 
\ Number of characters &  & $q-1$ & $q-1$ & $q-1$ \\ \hline 
\ Series &  & Principal series & Principal series & Principal series \\ \hline

\end{tabular}

\begin{tabular}{|c|c|c|}
\hline
\ $
\begin{array}{c}
{\text{Represen-}} \\ 
{\text{tatives}}%
\end{array}%
$ & ${\mathbf{\chi }}_{\alpha ,\beta }^{q^{2}+q+1}$ & ${\mathbf{\chi }}
_{\alpha ,\beta }^{q\left( q^{2}+q+1\right) }$ \\ \hline 
$T^{a}$ & $\left( q^{2}+q+1\right) \left( \alpha \beta ^{2}\right) \left(
a\right) $ & $q\left( q^{2}+q+1\right) \left( \alpha \beta ^{2}\right)
\left( a\right) $ \\ \hline
$T_{1}^{a}$ & $\left( q+1\right) \left( \alpha \beta ^{2}\right) \left(
a\right) $ & $q\left( \alpha \beta ^{2}\right) \left( a\right) $ \\ \hline
$T_{11}^{a}$ & $\left( \alpha \beta ^{2}\right) \left( a\right) $ & $0$ \\ \hline 
$T^{ab}$ & $\left( q+1\right) \alpha \left( a\right) \beta \left( ab\right)
+\beta ^{2}\left( a\right) \alpha \left( b\right) $ & $\left( q+1\right)
\alpha \left( a\right) \beta \left( ab\right) +q\beta ^{2}\left( a\right)
\alpha \left( b\right) $ \\ \hline
$T_{1}^{ab}$ & $\alpha \left( a\right) \beta \left( ab\right) +\beta
^{2}\left( a\right) \alpha \left( b\right) $ & $\alpha \left( a\right) \beta
\left( ab\right) $ \\ \hline
$T^{abc}$ & $\sum\limits_{\left( a,b,c\right) }\alpha \left( a\right) \beta
\left( bc\right) $ & $\sum \limits_{\left( a,b,c\right) }\alpha \left(
a\right) \beta \left( bc\right) $ \\   \hline
$T^{\varkappa a}$ & $\alpha \left( a\right) \beta \left( \varkappa \varkappa
^{q}\right) $ & $-\alpha \left( a\right) \beta \left( \varkappa \varkappa
^{q}\right) $ \\  \hline
$T^{z}$ & $0$ & $0$ \\ \hline
 
\ Number of characters & $\left( q-1\right) \left( q-2\right) $ & $\left(
q-1\right) \left( q-2\right) $ \\ \hline
\ Series & Principal series &  Principal Series \\ \hline 
 
\end{tabular}

\noindent 
\begin{tabular}{|c|c|c|c|}
\hline
\ $%
\begin{array}{c}
{\text{Represen-}} \\ 
{\text{tatives}}%
\end{array}%
$ & ${\mathbf{\chi }}_{\alpha ,\beta ,\gamma }^{\left( q+1\right) \left(
q^{2}+q+1\right) }$ & ${\mathbf{\chi }}_{\alpha ,\lambda }^{q^{3}-1}$ & ${%
\mathbf{\chi }}_{\varphi }^{\left( q-1\right) \left( q^{2}-1\right) }$ \\ \hline
$T^{a}$ & $\left( q+1\right) \left( q^{2}+q+1\right) \left( \alpha \beta
\gamma \right) \left( a\right) $ & $\left( q^{3}-1\right) \alpha \left(
a\right) \lambda \left( a\right) $ & $\left( q-1\right) \left(
q^{2}-1\right) \varphi \left( a\right) $ \\ \hline 
$T_{1}^{a}$ & $\left( 2q+1\right) \left( \alpha \beta \gamma \right) \left(
a\right) $ & $-\alpha \left( a\right) \lambda \left( a\right) $ & $-\left(
q-1\right) \varphi \left( a\right) $ \\ \hline
$T_{11}^{a}$ & $\left( \alpha \beta \gamma \right) \left( a\right) $ & $%
-\alpha \left( a\right) \lambda \left( a\right) $ & $\varphi \left( a\right) 
$ \\ \hline
$T^{ab}$ & $\left( q+1\right) \sum \limits_{\left( \alpha ,\beta ,\gamma
\right) }\left( \alpha \beta \right) \left( a\right) \gamma \left( b\right) $
& $\left( q-1\right) \alpha \left( b\right) \lambda \left( a\right) $ & $0$
\\ \hline
$T_{1}^{ab}$ & $\sum \limits_{\left( \alpha ,\beta ,\gamma \right) }\left(
\alpha \beta \right) \left( a\right) \gamma \left( b\right) $ & $-\lambda
\left( a\right) \alpha \left( b\right) $ & $0$ \\ \hline
$T^{abc}$ & $\sum \limits_{\left( \alpha ,\beta ,\gamma \right) }\alpha
\left( a\right) \beta \left( b\right) \gamma \left( c\right) $ & $0$ & $0$
\\ \hline
$T^{\varkappa a}$ & $0$ & $-\alpha \left( a\right) \left( \lambda +\lambda
^{q}\right) \left( \varkappa \right) $ & $0$ \\ \hline
$T^{z}$ & $0$ & $0$ & $\left( \varphi +\varphi ^{q}+\varphi ^{q^{2}}\right)
\left( z\right) $ \\ \hline
\  &  & $%
\begin{array}{l}
\lambda \in \left( {\mathbb{F}}_{q^{2}}^{\times }\right) ^{\wedge } \\ 
\lambda \neq \lambda ^{q}%
\end{array}
$ & $
\begin{array}{l}
\varphi \in \left( {\mathbb{F}}_{q^{3}}^{\times }\right) ^{\wedge } \\ 
\varphi \neq \varphi ^{q}%
\end{array}%
$ \\ \hline
\ $%
\begin{array}{c}
{\text{Number of}} \\ 
{\text{characters}}%
\end{array}%
$ & $\frac{1}{6}\left( q-1\right) \left( q-2\right) \left( q-3\right) $ & $%
\frac{1}{3}q\left( q-1\right) ^{2}$ & $\frac{1}{3}q\left( q^{2}-1\right) $
\\ \hline
\  &  Principal series & Intermediate series & Discrete series\\ \hline
 
\end{tabular}
 
\bigskip

The following lemma describes the decomposition of the character obtained
from an irreducible character in the "degenerate" case, i.e., when two or
more of the (distinct) parameters are now taken to be equal. \bigskip
 
\begin{lem}
We have 
\begin{enumerate}
\item $\chi _{\alpha \circ N_{3}}^{(q-1)(q^{2}-1)}=\chi _{\alpha }^{1}-\chi
_{\alpha }^{q^{2}+q}+\chi _{\alpha }^{q^{3}}$

\item $\chi _{\alpha ,\beta \circ {\mathbf{N}}_{2}}^{q^{3}-1}=\chi _{\alpha
,\beta }^{q^{3}+q^{2}+q}-\chi _{\alpha ,\beta }^{q^{2}+q+1}$

\item $\chi _{\alpha ,\alpha \circ N_{2}}^{q^{3}-1}=\chi _{\alpha
}^{q^{3}}-\chi _{\alpha }^{1}$

\item $\chi _{\alpha ,\alpha }^{q(q^{2}+q+1)}=\chi _{\alpha }^{q^{3}}+\chi
_{\alpha }^{q^{2}+q}$

\item $\chi _{\alpha ,\alpha }^{q^{2}+q+1}=\chi _{\alpha }^{1}+\chi _{\alpha
}^{q^{2}+q}$

\item $\chi _{\alpha ,\beta ,\beta }^{(q+1)(q^{2}+q+1)}=\chi _{\alpha ,\beta
}^{q^{2}+q+1}+\chi _{\alpha ,\beta }^{q(q^{2}+q+1)}$

\item $\chi _{\alpha ,\alpha ,\alpha }^{(q+1)(q^{2}+q+1)}=\chi _{\alpha
}^{1}+2\chi _{\alpha }^{q^{2}+q}+\chi _{\alpha }^{q^{3}}$
\end{enumerate}
\end{lem}
\begin{proof}
Follows from the character table of $G$.
\end{proof}

\section {Tensor Products of Irreducible Characters of the
Group $G.$}
\subsection{Description in terms of induced representations from torii}
In what follows, we denote by $\widetilde{\alpha }$ an extension to ${%
\mathbb{F}}_{q^{3}}^{\times }$ (or to ${\mathbb{F}}_{q^{2}}^{\times }$,
whichever is the case) of the character $\alpha \in \widehat{{\mathbb{F}}%
_{q}^{\times }}$ . On the other hand, 
for characters $\lambda $ and  $\mu$ of 
${\mathbb{F}}_{q^{3}}^{\times }$ (or of ${\mathbb{F}}_{q^{2}}^{\times }$), 
$\lambda _{\mid }$ denotes the
restriction to ${\mathbb{F}}_{q}^{\times }$ of the character $\lambda$ and 
 $\mu ^{q}$ stands for $\mu \circ F$, where $F$ is the Frobenius ${\mathbb{F}}_{q}^{\times }$-automorphism of ${\mathbb{F}}_{q^{3}}^{\times }$
(or ${\mathbb{F}}_{q^{2}}^{\times }$), given by $F(x)=x^{q}.$
 
\begin{thm}  With notations as above, we have:
\begin{enumerate}
\item 
\begin{enumerate}
\item[ i.] $\chi _{\alpha ,\lambda }^{q^{3}-1}\otimes \chi _{\beta ,\gamma
,\delta }^{\left
(q+1\right
)\left
(q^{2}+q+1\right
)}=Ind_{T_{m}}^{G}\left
(\gamma
\delta \lambda ,\alpha \beta \right )+  \chi _{\alpha \beta ,\widetilde {
\gamma \delta }\lambda }^{q^{3}-1}   + \\  +\left (\chi _{\alpha \gamma ,\widetilde {
\beta \delta }\lambda }^{q^{3}-1}+ 
\chi _{\alpha \delta ,\widetilde {\beta
\gamma }\lambda }^{q^{3}-1}\right ) \otimes  \chi _{1}^{q^{2}+q}$

\item[ ii.] $\chi _{\alpha \lambda ,(\alpha \circ {\mathbf{N}}_{2})\lambda
^{2}}^{q^{3}-1}\otimes \chi _{\beta ,\gamma ,\delta }^{\left( q+1\right)
\left( q^{2}+q+1\right) }=Ind_{T_{i}}^{G}\left( \alpha \lambda \beta ,\alpha
\lambda \gamma ,\alpha \lambda \delta \right) -\chi _{\alpha \lambda \beta
,\alpha \lambda \gamma ,\alpha \lambda \delta }^{\left( q+1\right) \left(
q^{2}+q+1\right) }$
\end{enumerate}

\item $\chi _{\alpha ,\lambda }^{q^{3}-1}\otimes \chi _{\beta ,\mu
}^{q^{3}-1}=Ind_{T_{m}}^{G}(\lambda \mu ,\alpha \beta )-\chi _{\alpha \beta
,\lambda \mu ^{q}}^{q^{3}-1}$

\item \begin{enumerate}
\item[ i.] $\chi _{\alpha ,\lambda }^{q^{3}-1}\otimes \chi _{\varphi
}^{(q-1)(q^{2}-1)}=Ind_{T_{a}}^{G}\left (\widetilde {\alpha }\widetilde {%
\lambda _{\mid }}\varphi \right )-\chi _{\ \widetilde {\alpha }\widetilde {%
\lambda _{\mid }}\varphi }^{(q-1)(q^{2}-1)}$

\item[ ii.] $\chi _{\alpha ,\lambda }^{q^{3}-1}\otimes \chi _{\varphi
}^{(q-1)(q^{2}-1)}=Ind_{T_{m}}^{G}\left (\lambda ,\alpha \varphi
\right
)+\chi _{\ \alpha \varphi ,\lambda }^{q^{3}-1}\otimes \left (\chi
_{1}^{1}-\chi _{1}^{q^{2}+q}\right )$
\end{enumerate}

\item $\chi _{\delta }^{q^{3}}\otimes \chi _{\alpha ,\beta ,\gamma }^{\left(
q+1\right) \left( q^{2}+q+1\right) }=Ind_{T_{i}}^{G}\left( \alpha \delta
,\beta \delta ,\gamma \delta \right) $

\item $\chi _{\alpha }^{q^{3}}\otimes \chi _{\varphi
}^{\left
(q-1\right
)\left (q^{2}-1\right )}=Ind_{T_{a}}^{G}\left (\alpha
\circ {\mathbf{N}}_{3}\right )\varphi $\smallskip

\item $\chi _{\alpha }^{q^{3}}\otimes \chi _{\beta ,\lambda
}^{q^{3}-1}=Ind_{T_{m}}^{G}\left (\left (\alpha \circ {\mathbf{N}}%
_{2}\right
)\lambda ,\alpha \beta \right )$ \smallskip

\item 
\begin{enumerate}
\item[ i.] $\chi _{\alpha }^{q^{3}}\otimes \chi _{\beta
}^{q^{3}}=Ind_{T_{m}}^{G}\left (\alpha \beta \circ {\mathbf{N}}_{2},\alpha
\beta \right )+\chi _{\alpha \beta }^{q^{3}}$

\item[ ii.] $\chi _{\alpha }^{q^{3}}\otimes \chi _{\beta
}^{q^{3}}=Ind_{T_{i}}^{G}\left( \alpha \beta ,\alpha \beta ,\alpha \beta
\right) -\chi _{\alpha \beta }^{q^{3}}\otimes \left( 2\chi
_{1}^{q^{2}+q}+\chi _{1}^{1}\right) $

\item[ iii.] $\chi _{\alpha }^{q^{3}}\otimes \chi _{\beta
}^{q^{3}}=Ind_{T_{a}}^{G}\alpha \beta \circ {\mathbf{N}}_{3}+\chi _{\alpha
\beta }^{q^{3}}\otimes \left (\chi _{1}^{q^{2}+q}-\chi _{1}^{1}\right )$
\end{enumerate}

\item 
\begin{enumerate}
\item[ i.] $\chi _{\varphi }^{(q-1)(q^{2}-1)}\otimes \chi _{\alpha ,\beta
,\gamma
}^{\left
(q+1\right
)\left
(q^{2}+q+1\right
)}=Ind_{T_{a}}^{G}\varphi
\left (\widetilde {\alpha \beta \gamma }\right
)+\chi _{\varphi \left (%
\widetilde {\alpha \beta \gamma }\right
)}^{(q-1)(q^{2}-1)}\otimes \\
\left
(2\chi _{1}^{q^{2}+q}+\chi _{1}^{1}\right )$

\item[ ii.] $\chi _{\varphi }^{(q-1)(q^{2}-1)}\otimes \chi _{\alpha ,\beta
,\gamma }^{\left( q+1\right) \left( q^{2}+q+1\right) }=Ind_{T_{i}}^{G}\left(
\alpha ,\beta ,\varphi \gamma \right) +\chi _{\alpha ,\beta ,\varphi \gamma
}^{\left( q+1\right) \left( q^{2}+q+1\right) }\\ 
\otimes \left( \chi
_{1}^{1}-\chi _{1}^{q^{2}+q}\right) $
\end{enumerate}

\smallskip

\item $\chi _{\varphi }^{(q-1)(q^{2}-1)}\otimes \chi _{\psi
}^{(q-1)(q^{2}-1)}=$

$Ind_{T_{a}}^{G}\varphi \psi +\chi _{\varphi \psi
^{q}}^{(q-1)(q^{2}-1)}+\chi _{\varphi \psi ^{q^{2}}}^{(q-1)(q^{2}-1)}-\chi
_{\varphi \psi }^{(q-1)(q^{2}-1)}\otimes \left (\chi _{1}^{q^{2}+q}+\chi
_{1}^{1}\right )$ \smallskip

\item $\chi _{\alpha ,\beta ,\gamma
}^{\left
(q+1\right
)\left
(q^{2}+q+1\right )}\otimes \chi _{\delta
,\varepsilon ,\eta }^{\left
(q+1\right )\left (q^{2}+q+1\right )}=$

$Ind_{T_{i}}^{G}\left( \alpha \delta ,\beta \varepsilon ,\gamma \eta \right)
+\sum\limits_{\left( \delta \varepsilon \eta \right) \in S_{3}-\left\{
I\right\} }\chi _{\alpha \delta ,\beta \varepsilon ,\gamma \eta }^{\left(
q+1\right) \left( q^{2}+q+1\right) }+$

$\left (\chi _{\alpha \varepsilon ,\beta \eta ,\gamma \delta
}^{\left
(q+1\right )\left (q^{2}+q+1\right )}+\chi _{\alpha \eta ,\beta
\delta ,\gamma \varepsilon }^{\left
(q+1\right
)\left
(q^{2}+q+1\right
)}\right
)\otimes \left (\chi
_{1}^{q^{2}+q}-2\chi _{1}^{1}\right )$
\end{enumerate}
\end{thm} 
\begin{proof}
The proof consists in computing the relevant characters on the different
conjugacy classes of $G$ (See section 2). Specifically, the computations
rely on lemma 1 and the character table of $G.$

As an example of the above, we present a proof of 8.i.

\begin{enumerate}
\item[ a)] In the case of the conjugacy classes given by $%
T^{a},T_{1}^{a},T_{11}^{a},T^{ab},T_{1}^{ab},T^{abc}$ and $T^{\varkappa a}$,
the result follows directly from both the character table and lemma 1.

\item[ b)] We have

$\left( \chi _{\varphi }^{(q-1)(q^{2}-1)}\otimes \chi _{\alpha ,\beta
,\gamma }^{\left( q+1\right) \left( q^{2}+q+1\right) }\right) (T^{z})=\left[
\left( \varphi +\varphi ^{q}+\varphi ^{q^{2}}\right) (z)\right] 0=0$

Also, using the character table and part 7 of Lemma 1, we have 

$\left(
Ind_{T_{a}}^{G}\varphi \left( \widetilde{\alpha \beta \gamma }\right) +\chi
_{\varphi \left( \widetilde{\alpha \beta \gamma }\right)
}^{(q-1)(q^{2}-1)}\otimes \left( 2\chi _{1}^{q^{2}+q}+\chi _{1}^{1}\right)
\right) (T^{z})=$

$\dfrac{1}{\left\vert T_{a}\right\vert }\sum \limits_{\substack{ X\in G  \\ 
XT^{z}X^{-1}\in T_{a}}}\varphi \left( \widetilde{\alpha \beta \gamma }
\right) \left( XT^{z}X^{-1}\right) -   \varphi \left( \widetilde{\alpha
\beta \gamma }\right)(z)  -  \left( \varphi \left( \widetilde{\alpha \beta \gamma }
\right) \right) ^{q} (z)  -   \left( \varphi \left( \widetilde{\alpha \beta \gamma } \right) \right) ^{q^{2}} (z)$

Now, by part 5 of lemma 1, the above expression becomes

$\dfrac{1}{\left\vert T_{a}\right\vert }\sum \limits_{X\in T_{a}\times
\Gamma _{3}}\varphi \left( \widetilde{\alpha \beta \gamma }\right) \left(
XT^{z}X^{-1}\right) -\varphi \left( \widetilde{\alpha \beta \gamma }\right)
(z)-\left( \varphi \left( \widetilde{\alpha \beta \gamma }\right) \right)
^{q}(z)-    $

$  \ \left( \varphi \left( \widetilde{\alpha \beta \gamma }\right)
\right) ^{q^{2}}(z) =
\dfrac{\left\vert T_{a}\right\vert }{\left\vert T_{a}\right\vert }\left(
\varphi \left( \widetilde{\alpha \beta \gamma }\right) (z)+\varphi \left( 
\widetilde{\alpha \beta \gamma }\right) (z^{q})+\varphi \left( \widetilde{%
\alpha \beta \gamma }\right) (z^{q^{2}})\right) $ 

$-  \varphi \left( \widetilde{\alpha \beta \gamma }\right) (z)-\left( \varphi
\left( \widetilde{\alpha \beta \gamma }\right) \right) ^{q}(z)-\left( \
\varphi \left( \widetilde{\alpha \beta \gamma }\right) \right) ^{q^{2}}(z)=0$%
.
\end{enumerate}
\end{proof}

Finally, using lemma 2, the next corollary address the case of
irreducible representations that arise from limit cases of previous formulas
where now some parameters coincide.

\begin{cor}

With notations as above, we have

\begin{enumerate}
\item $\chi _{\alpha }^{q^{3}}\otimes \left[ \chi _{\beta ,\gamma
}^{q^{2}+q+1}+\chi _{\beta ,\gamma }^{q\left( q^{2}+q+1\right) }\right]
=Ind_{T_{i}}^{G}\left( \alpha \beta ,\alpha \gamma ,\alpha \gamma \right) $

\item $\chi _{\alpha }^{q^{3}}\otimes \chi _{\beta }^{q^{2}+q}=2\chi
_{\alpha \beta }^{q^{3}}+Ind_{T_{m}}^{G}\left (\alpha \beta \circ {\mathbf{N}%
}_{2},\alpha \beta \right )-Ind_{T_{a}}^{G}\left (\alpha \circ {\mathbf{N}}%
_{3}\right )\left (\beta \circ {\mathbf{N}}_{3}\right )$
\end{enumerate}
\end{cor} 
\begin{proof}
The first formula follows from 4 of theorem 1 above, taking two parameters to be
equal. To prove the second, it is enough to use parts 7 and 5 of theorem 1 and to reduce terms
applying part 1 of Lemma 2.
\end{proof}

\subsection{Description in terms of Gelfand-Graev representations}

The next proposition describes the tensor product of a discrete and a principal series representation of $G$ as the sum of the classical  and a generalized Gelfand-Graev representation of  $G$. 
 
We denote by $N$ the standard unipotent subgroup of $G$, and by $N_{2}$ the
subgroup of $N$ whose $(1,2)$ entry is $0,$ i.e., 

$\ \ \ \ \ \ \ \ N=\left\{
\left( 
\begin{array}{ccc}
1 & x & z \\ 
0 & 1 & y \\ 
0 & 0 & 1%
\end{array}%
\right) |\;x,y,z\in k\right\} ,$ $\ \ \ N_{1}=\left\{ \left( 
\begin{array}{ccc}
1 & 0 & z \\ 
0 & 1 & y \\ 
0 & 0 & 1%
\end{array}%
\right) |\;y,z\in k\right\} $.

For a non trivial character $\varphi $ of $k^{+}$, we also denote by $
\varphi $ the character on $N$ defined by $\ \ \ \ \  \ \ \ \varphi (\left ( 
\begin{array}{ccc}
1 & x & z \\ 
0 & 1 & y \\ 
0 & 0 & 1%
\end{array}
\right ))=\varphi \left (x+y\right ).$

\begin{prop} 
With notations as above, we have

$\chi _{\psi }^{(q-1)(q^{2}-1)}\otimes \chi _{\beta ,\gamma ,\delta
}^{\left( q+1\right) \left( q^{2}+q+1\right) }=Ind_{ZN}^{G}(\psi \beta
\gamma \delta )\varphi +Ind_{ZN_{1}}^{G}(\psi \beta \gamma \delta )\varphi $
\end{prop}
\begin{proof}
We compute first the involved characters on the conjugacy class of $%
T_{1}^{a}: $

$\left[ \chi _{\psi }^{(q-1)(q^{2}-1)}\otimes \chi _{\beta ,\gamma ,\delta
}^{\left( q+1\right) \left( q^{2}+q+1\right) }\right]
(T_{1}^{a})=-(q-1)(2q+1)\left( \psi \beta \gamma \delta \right) \left(
a\right) .$

Let $D_{1}=\left\{ X\;|\;XT_{1}^{a}X^{-1}\in ZN_{1}\right\} =\left\{ \left( 
\begin{array}{ccc}
p & q & r \\ 
l & m & n \\ 
0 & t & 0%
\end{array}%
\right) \in G\right\} ,$ and

let $D_{2}=\left\{ \left( 
\begin{array}{ccc}
p & q & r \\ 
0 & m & n \\ 
0 & t & s%
\end{array}%
\right) \;|\;s\neq 0,p\neq 0,ms-tn\neq 0\right\} $, we see that,

$D=\left\{ X\;|\;XT_{1}^{a}X^{-1}\in ZN\right\} =D_{1}\cup D_{2}=\left\{
\left( 
\begin{array}{ccc}
p & q & r \\ 
l & m & n \\ 
0 & t & s%
\end{array}%
\right) \in G\;|\;ls=0\right\} $,

so

$\left[ Ind_{ZN}^{G}(\psi \beta \gamma \delta )\varphi
+Ind_{ZN_{1}}^{G}(\psi \beta \gamma \delta )\varphi \right] \left(
T_{1}^{a}\right) =\\
\dfrac{1}{\left\vert ZN\right\vert }\sum \limits_{X\in
D}(\psi \beta \gamma \delta )\varphi \left( XT_{1}^{a}X^{-1}\right) 
+\dfrac{1%
}{\left\vert ZN_{1}\right\vert }\sum \limits_{X\in D_{1}}(\psi \beta \gamma
\delta )\varphi \left( XT_{1}^{a}X^{-1}\right) = $

$=\left[ \dfrac{1}{\left\vert ZN\right\vert }+\dfrac{1}{\left\vert
ZN_{1}\right\vert }\right] \sum \limits_{\substack{ t\in {\mathbb{F}}%
_{q}^{\times };q,m\in {\mathbb{F}}_{q}  \\ pn-lr\neq 0}}(\psi \beta \gamma
\delta )\left( a\right) \varphi \left( \dfrac{a^{-1}l}{t}\right)
+ \\
+  \dfrac{1}{\left\vert ZN\right\vert }  \sum \limits_{\substack{ s,p\in {\mathbb{F}}
_{q}^{\times };r,q\in {\mathbb{F}}_{q}  \\ ms-tn\neq 0}}(\psi \beta \gamma
\delta )\left( a\right) \varphi \left( \dfrac{a^{-1}ps}{ms-tn}\right) =$

$=\left[ \dfrac{q+1}{q^{3}\left( q-1\right) }\left( -q^{3}(q-1)^{2}\right) +%
\dfrac{1}{q^{3}\left( q-1\right) }\left( -q^{4}(q-1)^{2}\right) \right]
(\psi \beta \gamma \delta )\left( a\right) $

$=-(q-1)(2q+1)\left (\psi \beta \gamma \delta \right )\left (a\right ).$

The computations on the conjugacy class of $T_{11}^{a}$ is similar to the
above one. In this case the sum corresponding to $ZN_{1}$ has no support.

The remaining cases are straightforward.
\end{proof}

The above proposition suggests that   for $G = GL(n, \mathbb F_{q})$  
  the tensor product of a cuspidal (discrete) and a principal
series representation may be expressed as the direct sum of  $ \frac{(n-2)(n-1)}{2} + 1  $   
packets of generalized Gelfand-Graev representations of the same dimension,  
the first one consisting simply of the  classical Gelfand-Graev representation and the last one consisting only of the generalized Gelfand-Graev representation induced from   $Z$  times the upper unipotent subgroup  with  all non zero upper diagonal entries in  the last column.

To describe our general conjecture more precisely we introduce some notations:

Let   $G = GL(n, \mathbb F_{q})  $ and  $Z$ be the center of  $G.$  Let  $\alpha$ be any character of  $Z$ and  $\varphi$ any non trivial character of    $\mathbb F_{q}^{+} $. We extend as usual  $\varphi$  to a character, still denoted by  $\varphi,$ of the standard upper unipotent subgroup $N,$ or any of its subgroups,    by    $ \varphi(u_{ij}) = \varphi (u_{12}+u_{23}+ \cdots + u_{(n-1)n}) $ for  $u = (u_{ij}) \in N$.

For any subgroup  $N'$ of    $N$, we put
$$  \Gamma_{N'}(\alpha\varphi) =  Ind_{ZN'}^{G}(\alpha\varphi). $$

Let   $  \tilde n   =   1 + 2 + \cdots + (n-2)  = \frac{(n-2)(n-1)}{2}.    $ 

We define the family of numbers  $ c_{j}(n) $ for $0 \leq j \leq \tilde n $ by
$$   (q+1)(q^{2}+q+1) \cdots  +(q^{n-2}+q^{n-1}+ \cdots+1) = \sum _{0 \leq j \leq \tilde n} c_{j}(n)q^{j}$$

We will call  {\em Gelfand-Graev interpolating family}, any family  $\{{\mathcal N}_{i}\}_{0 \leq j \leq \tilde n}$ of sets of subgroups of the standard upper unipotent subgroup  $N$ of  $G$ such that, for all $i$:
 
\begin{enumerate}
\item[a.]   $|{\mathcal N}_{i}| = c_{i}(n) $
\item  [b.] each subgroup $N' \in {\mathcal N}_{i}$ is defined by the vanishing of  $i$ upper unipotent entries, so that  $ [N:N'] = q^{i}.  $  
 
\end{enumerate}

  Notice that  ${\mathcal N}_{0}$ consists only of  $ N $   and   $ {\mathcal N}_{\tilde n} $  contains only one  subgroup, of order $q^{(n-1)}$. 
  
\begin{con}
Fix a non trivial character  $\varphi $ of  $ \mathbb F_{q}^{+} $. There exists a Gelfand-Graev interpolating family  $\{{\mathcal N}_{i}\}_{0 \leq j \leq \tilde n}$  such that

$$ \chi _{\psi }^{(q-1)(q^{2}-1) \cdots (q^{n-1}-1)}\otimes \chi _{\beta
_{1},...,\beta_{n}}^{\left(q+1\right)\left(q^{2}+q+1\right)..(q^{n-1}+q^{n-2}+...+q+1)}= $$ 
  $$ = \bigoplus_{0 \leq j \leq \tilde n} [\bigoplus_{N' \in {{\mathcal N}_{i}} }   \Gamma_{N'}((\psi \beta _{1}...\beta _{n})\varphi ) ]$$

\end{con}
 
\section{Application: Clebsch-Gordan coefficients for the tensor product of
two cuspidal representations of $G$.}

Finally, we notice
that the decomposition of tensor products of irreducible representations in
irreducible constituents can be easily computed, once you have described
these tensor products in terms of induced representations.  

For example, consider the case of \ the tensor product \ $\chi _{\varphi
}^{(q-1)(q^{2}-1)}\otimes \chi _{\psi }^{(q-1)(q^{2}-1)}$ of \ two cuspidal
(discrete series) representations of  $G.$

If \ $\left \langle \Omega ,\Theta \right \rangle _{G}$ stands for the usual
inner product of two complex valued functions on $G,$ i.e.,

$$\left \langle \Omega ,\Theta \right \rangle _{G}=\dfrac{1}{\left \vert
G\right \vert }\sum _{t\in G} \Omega (t)\overline {\Theta (t)},$$ 
where $\ \overline {\Theta (t)}$ \ is the conjugate of the complex number $%
\Theta (t)$, \ then we have \ that the multiplicities of the irreducible
representations of \ $G$ \ in this tensor product are given as follows:

\begin{enumerate}
\item $\left \langle \chi _{\varphi }^{(q-1)(q^{2}-1)}\otimes \chi _{\psi
}^{(q-1)(q^{2}-1)},\chi _{\alpha }^{1}\right \rangle =1$ if and only if $\
\varphi \psi =\alpha \circ {\mathbf{N}}_{3}$

\item $\left \langle \chi _{\varphi }^{(q-1)(q^{2}-1)}\otimes \chi _{\psi
}^{(q-1)(q^{2}-1)},\chi _{\alpha ,\beta ,\gamma
}^{\left
(q+1\right
)\left
(q^{2}+q+1\right )}\right \rangle
_{G}=\left
\langle \varphi \psi ,\alpha \beta \gamma \right \rangle
_{k^{\times }}$

\item $\left \langle \chi _{\varphi }^{(q-1)(q^{2}-1)}\otimes \chi _{\psi
}^{(q-1)(q^{2}-1)},\chi _{\Lambda }^{(q-1)(q^{2}-1)}\right \rangle =$
\end{enumerate}

$\left \langle \varphi \psi ,\Lambda \right \rangle _{K^{\times
}}+\left
\langle \varphi \psi ,\Lambda ^{q}\right \rangle _{K^{\times
}}+\left
\langle \varphi \psi ,\Lambda ^{q^{2}}\right \rangle _{K^{\times
}}+(q-3)\left \langle \varphi \psi ,\Lambda \right \rangle _{k^{\times
}}+ $

$ \left \langle \chi _{\varphi \psi ^{q}}^{(q-1)(q^{2}-1)},\chi _{\Lambda
}^{(q-1)(q^{2}-1)}\right \rangle +\left \langle \chi _{\varphi \psi
^{q^{2}}}^{(q-1)(q^{2}-1)},\chi _{\Lambda }^{(q-1)(q^{2}-1)}\right \rangle 
= $

$\left \langle \varphi \psi ,\Lambda \right \rangle _{K^{\times }}+\langle
\varphi \psi ,\Lambda ^{q}\rangle _{K^{\times }}+\left \langle \varphi \psi
,\Lambda ^{q^{2}}\right \rangle _{K^{\times }}+\left \langle \varphi \psi
^{q},\Lambda \right \rangle _{K^{\times }} + 
\left \langle \varphi \psi
^{q},\Lambda ^{q}\right \rangle _{K^{\times }}  
+\left \langle \varphi \psi
^{q},\Lambda ^{q^{2}}\right \rangle _{K^{\times }}$   

$ + \left \langle \varphi \psi ^{q^{2}},\Lambda \right \rangle _{K^{\times
}}+\left \langle \varphi \psi ^{q^{2}},\Lambda ^{q}\right \rangle
_{K^{\times }}+\left \langle \varphi \psi ^{q^{2}},\Lambda
^{q^{2}}\right
\rangle _{K^{\times }}+(q-3)\left \langle \varphi \psi
,\Lambda \right
\rangle _{k^{\times }}$

\end{document}